\documentclass{amsart}
\usepackage{mathtools,amsmath,amssymb,amsthm,array,etoolbox}
\usepackage{color}
\usepackage{cases}
\usepackage{ulem}
\usepackage{lineno}
\usepackage{amsmath}
\usepackage{cancel}

\numberwithin{equation}{section}

\newtheorem{theorem}{Theorem}[section]
\newtheorem{lemma}[theorem]{Lemma}

\theoremstyle{definition}
\newtheorem{definition}[theorem]{Definition}

\newtheorem{proposition}[theorem]{Proposition}
\theoremstyle{remark}
\newtheorem{remark}[theorem]{Remark}

\newcommand{\R}{\mathbb{R}}

\newcommand{\be}{\begin{equation}}
	\newcommand{\ee}{\end{equation}}
\newcommand{\bse}{\begin{subequations}}
	\newcommand{\ese}{\end{subequations}}
\def\bea{\begin{eqnarray}}
	\def\eea{\end{eqnarray}}
 \allowdisplaybreaks
\begin{document}

\title{On a Problem Posed by Brezis and Mironescu}
\author{Fanghua Lin, Malkiel Shoshan, }
\address{Courant Institute of Mathematical Sciences, New York University, NY 10012, USA}
\email{linf@cims.nyu.edu}
\author{Changyou Wang}
\address{Department of Mathematics, Purdue University, West Lafayette, IN 47907, USA}
\email{wang2482@purdue.edu}
\begin{abstract}
The purpose of this note is to present a positive answer to an open problem proposed in the recent book \cite{Brezis-Mironescu} by H. Brezis and P. Mironescu. It has been stated in this book 
{\it Sobolev Maps to the Circle} as Proposition 4.3. We demonstrate, in particular, the value of the least mass of the area minimizing integral rectifiable currents with a given boundary equals to the infimum of areas among smoothly immersed submanifolds with the same boundary, under the assumption that the boundary is that of a smooth submanifold. 
\end{abstract}
\date{}
\maketitle
\centerline{\it Dedicated to the memory of Professor Haim Brezis with admiration}

\section{Introduction}
In {their book} \textit{Sobolev Maps to the Circle} \cite[Section 4.2]{Brezis-Mironescu}, 
Brezis  and Mironescu 
introduced a ``distinguished class of currents". Let $\Omega \subset \mathbb{R}^N$, $N \geq 2$, be 
a bounded convex open set. Given an integer $l$, $0 \leq l \leq N-2$, we set

\begin{align*}\mathcal{F}^l &= \Big\{T \ \Big|\ T \mbox{ is a closed smooth $l$-dimensional manifold such that } 
 T= \partial \Gamma\\
 &\qquad\quad\text{ for some } (l+1)\text{-integral rectifiable current } \Gamma \subset\Omega
 \ \text{with}\  \mathbf{M}(\Gamma)<\infty \Big\},
 \end{align*}
 {where $\mathbf{M}(\Gamma)$ denotes the mass of the current $\Gamma$}.
For $T \in \mathcal{F}^l$, we set

\begin{align}\label{mass1} A^l(T)= \inf \Big\{\mathbf{M}(\Gamma)\Big|
\Gamma \subset \Omega \text{ is } (l+1)\text{-integral rectifiable current,}\ \partial\Gamma = T \Big\}.
\end{align}

We also consider the subset $\mathcal{F}_0^l,\  \subset \mathcal{F}^l,$\  {that is defined by}
\begin{align*}
\mathcal{F}_0^l &= \Big\{T\in \mathcal{F}^l \ \Big|\ T = \partial M \text{ for some smoothly immersed and oriented}\\
& \qquad\qquad\quad\ \ (l+1)\text{-dimensional submanifold } M {\subset \Omega\times \R},\ \text{with}\  |M|<\infty \Big\},
\end{align*}
{{where $\big|M\big|$ denotes area of the submanifold $M$}}. We also set, for $T \in \mathcal{F}_0^l$,\\
\begin{align}\label{mass2}
A_0^l(T) &= \inf \Big\{ \big|M\big|\ \Big|\ M \text{ is a } (l+1)\text{-dimensional smoothly immersed}\\
& \qquad\qquad\qquad\text{and oriented submanifold in } \Omega\times \R,
\text{ with } \partial M = T \Big\}.\nonumber
\end{align}

Brezis and Mironescu \cite{Brezis-Mironescu} 
claim that $T$ being a smooth boundary of some smoothly immersed {$(l+1)$-dimensional}
submanifold is  sufficient to ensure the following statement. 

\begin{theorem} \label{equal} {\it For any $T \in \mathcal{F}_0^l$, $0\le l\le N-2$, the two definitions of  least mass spanned by $T$ are equivalent, i.e. minimizing the areas of smoothly immersed and oriented  $(l+1)$-dimensional submanifolds in $\Omega\times \R$ whose boundary is $T$ is equivalent to minimizing the masses of {$(l+1)$-}integral rectifiable currents whose boundary is $T$, {or equivalently},
 \[
 A_0^l(T) = A^l(T).
 \]}
 \end{theorem}
 
 \begin{remark} {{We would like to point out that Theorem 1.1 is slightly different from the Proposition 4.3 of \cite{Brezis-Mironescu}, where $\mathcal{F}_0^l$ 
 satisfies a stronger assumption, namely, $T\in \mathcal{F}_0^l$ must be
 the boundary of some smooth immersed $(l+1)$-dimensional manifold in $\Omega$;  while for $T\in\mathcal{F}_0^l$ $A_0^l(T)$  is defined to be
 the infimum of areas of smooth immersed $(l+1)$-dimensional manifolds in $\Omega$, rather than in $\Omega\times\R$ which is what we assume in Definition \ref{mass2}. Of course, area minimizing currents with the same boundary must be supported in $\Omega$ by the convex hull property.  On the other hand, our simple construction of smooth ones which approximate the area-minimizing ones in both the total mass and the flat norm need to be in  $\Omega\times\R$. In fact, the factor $\R$ can be replaced by an arbitrarily small interval 
$(0,  \varepsilon)$ under the same assumption as in the Proposition 4.3 of \cite{Brezis-Mironescu}. As one referee suggested that if one applies a more general and more involved construction by
\cite{Almg}, one may conclude the exact statement as the Proposition 4.3 of \cite{Brezis-Mironescu}. For simplicity, we leave Theorem 1.1 as what it is. }}
 \end{remark} 
 
 {Brezis-Mironsecu presented a proof of Proposition 4.3 in \cite {Brezis-Mironescu} for the cases $l = 0, N-2$ and attributed the intermediate cases $l=1,\cdots, N-3$ to the first author of this paper.
 Here, we provide a proof of the result as stated in Theorem 1.1 above.}
 
\smallskip
\indent Let $T \in \mathcal{F}^l_0$, {and let } $M$  be a smoothly immersed and oriented $(l+1)$-{dimensional }
submanifold {{in} $\Omega\times\R$ such that $T=\partial M$}, and let $\Gamma {\subset\Omega}$ be an $(l+1)$-integral rectifiable area minimizing current with $\partial \Gamma = T$ that realizes the infimum in the definition of $A^l(T),$
{i.e. $\mathbf{M}(\Gamma)=A^l(T)$}. The existence of such an integral rectifiable current {$T$}
is a well-known {result} in the geometric measure theory, due to Federer and Fleming \cite{Federer}.

\indent {Let $$\Pi(x, x')=x:\Omega\times \R\to\Omega$$
be the orthogonal projection map. Then 
it is easy to see that $\overline{M}=\Pi(M)\subset\Omega$ is an $(l+1)$-integral rectifiable current satisfying $\partial \overline{M}=T$.
Moreover, since $\Pi$ has Lipschitz norm $1$, it follows that the area of $M$ is no less than the mass of $\overline{M}$, i.e.,
\[
\big|\overline{M}\big|\le \big|M\big|.
\]
}
{From this, we} clearly have $A_0^l(T) \geq A^l(T)$. 

To show the reverse direction {$A_0^l(T) \leq A^l(M)$},  it suffices to show that for any $\varepsilon >0$, we can find a smoothly immersed and oriented $(l+1)$-submanifold $\Gamma_0{\subset\Omega\times\R}$ such that {{$\partial\Gamma_0=T$, and}}
\begin{equation}\label{reverse1}
 \mathbf{M}(\Gamma_0) <A^l(T)+\varepsilon= \mathbf{M}(\Gamma) +\varepsilon .
 \end{equation}
\indent The basic idea  to prove \eqref{reverse1} is to
construct $\Gamma_0$ from the original $\Gamma$ by cutting out a small tubular neighborhood  
of the singular set $\mathcal{S}$ of $\Gamma$, and pasting in $\mathcal{S}_\varepsilon$ which is something 
{\it smoothable} whose  area is also small, see Section 3 below.

\indent A crucial result needed here is the partial regularity theorem, due to Almgren \cite{Almgren}, 
that asserts that the singular set $\mathcal{S}$ of an area minimizing 
{$(l+1)$-integral rectifiable} current has Hausdorff dimension at most $(l-1)$. Using this we will show that a tubular neighborhood 
$E'$ of the singular set $\mathcal{S}$ will have small $(l+1)${-dimensional} Hausdorff measure.

We construct $\mathcal{S}_\varepsilon$ by using spherical inversion (see Section 4 below) to get a projection of $B = (\Gamma \setminus (E'\cup\mathcal{S})) \cup T \cup M_0 $ into a small sphere. The sphere will contain a scaled down inverted ``image'' $\widetilde{B}$ of $ B$ which will have small volume. $B$ will have ``holes'' where we cut out a tubular neighborhood, and the image $\widetilde{B}$ will also have holes. We connect $B$ and $\widetilde{B}$ through a cone $C$
(see Section 5) through the respective holes, so $\mathcal{S}_\varepsilon = C\cup \widetilde {B}$, and the resulting surface $(\Gamma \setminus (E'\cup\mathcal{S})) \cup T \cup \mathcal{S}_\varepsilon$ is {\it smoothable} and result in a smoothly immersed and oriented submanifold.

Lastly, we present in Section 7 an example to show why we need to take the infimum over smooth manifolds instead of the minimum. There exists 
a $T{\in\mathcal{F}_0^l}$ such that the infimum $A^l_0(T)$ is not achieved.

\section{Notation and Preliminaries}
We let $\Omega \subset \mathbb{R}^N$ be {a bounded convex} open set throughout. 
{For $0\le m\le N$, l}et $\mathcal{D}^m(\Omega)$ 
{denote}  the space of smooth, compactly supported \textit{m}-forms on $\Omega$,  that {is,}
\[ \mathcal{D}^m(\Omega) = \Big\{ \omega = \sum_{\alpha\in \mathbb{I}_{m, N}} \omega_\alpha dx^\alpha \ \Big| \  \omega_\alpha \in C_c^\infty(\Omega) \Big\} ,\]
{equipped} with the topology 
\[ \omega^k = \sum_{\alpha \in \mathbb{I}_{m,N}} \omega^k_\alpha dx^\alpha \;  \longrightarrow \; \omega = \sum_{\alpha \in \mathbb{I}_{m,N}} \omega_\alpha dx^\alpha  \]
assuming the existence of a compact set $K \subset \Omega$ such that $\text{spt} (\omega^k_\alpha) \subset K$ for all 
$\alpha \in \mathbb{I}_{m, N}, k \geq 1$, and $\displaystyle\lim_{k\to\infty} D^\beta \omega^k_\alpha = D^\beta \omega_\alpha$ for all $\alpha \in \mathbb{I}_{m,N}$ and all multi-index $\beta$. 

An $m$-dimensional \text{current} $T$ in $\Omega$ is {defined to be }a continuous linear functional on $\mathcal{D}^m(\Omega)$. 
More precisely, $T:\mathcal{D}^m(\Omega) \to \mathbb{R}$ is continuous in the sense that 
$$\big\langle T, \omega^k \big\rangle \xrightarrow{k\nearrow\infty} \big\langle T, \omega \big\rangle,$$ 
where $\omega^k \in \mathcal{D}^m(\Omega)$ is such that there exists an open set $U\Subset\Omega$ with 
$\text{spt}(\omega^k) \subset U$ for every $k$, and $\omega^k \to \omega$ in $\mathcal{D}^m(\Omega)$.

The support of a {$m$-dimensional} current $T$ 
is defined by $\text{spt} T = \Omega \setminus \bigcup V$, where $V \Subset\Omega$ is 
such that $\langle T, \omega \rangle =0$ for all $\omega \in \mathcal{D}^m(\Omega)$ with $\text{spt} (\omega) \subset V$.

We can define the boundary of an {$m$-dimensional} current $T$, {$\partial T$ as an $(m-1)$-
dimensional current,} in accordance with Stokes theorem, by 
\[ \big\langle \partial T, \tilde{\omega}\big\rangle  =\big \langle  T, d \tilde{\omega} \big\rangle, \ \ \forall\tilde{\omega}\in \mathcal{D}^{m-1}(\Omega), \]
where 
\[d \tilde{\omega} = \sum_{j=1}^N \sum_{\alpha \in \mathbb{I}_{m-1,N} } \frac{\partial \tilde{\omega}_\alpha}{\partial x^j} dx^j \wedge dx^\alpha
\in \mathcal{D}^m(\Omega) \]
is the exterior derivative of $\tilde{\omega}$.
The mass of an {$m$-dimensional} current $T$ is defined by 
\[\mathbf{M}(T) = \sup_{|\omega| \leq 1, \omega \in \mathcal{D}^m(\Omega)} \big\langle T, \omega \big\rangle.\]
Given a smooth oriented $m$-submanifold $M \subset \mathbb{R}^N$, {it induces an $m$-dimensional current [M] by}
$$\big\langle [M], \omega \big\rangle = \int_M \omega, \ \ \omega\in \mathcal{D}^m(\R^N),$$
then {the mass of $[M]$ equals to the $m$-dimensional Hausdorff measure of $M$, that is} $\big|[M]\big| = \mathcal{H}^m(M)$.

In 1959, Federer and Fleming \cite {Federer60} identified a class of currents which they deemed valuable in the study of calculus of variations, due to their compactness properties. These are the integral rectifiable currents.

\begin{definition} \textit{Let $k$ be an integer with $0 \leq k \leq N-1$, and $T$ be an {$k$-dimensional} 
current with finite mass. $T$ is an $k$-\text{integer rectifiable current} in $\mathbb{R}^N$ if there exist a sequence of oriented 
{$k$-dimensional} submanifolds $\Sigma_i$ and a sequence of Borel subset $K_i$ such that} $K_i \subset \Sigma_i $, 
$\displaystyle\sum_i \textit{H}^k(K_i) < \infty$, \textit{and}
\[ \langle T, \omega \rangle = \sum_i \int_{K_i} \omega, \ \ \  \forall \omega \in \mathcal{D}^m(\R^N) .\]
\end{definition}
In the case that $T$ is a $k$-rectifiable current with $ K_i$'s as above, for an $\mathcal{H}^k$-measurable set $A$, 
we will take {$T \lfloor A$} to be the $k$-rectifiable current defined by
\[ \big\langle T\lfloor A,\omega \big\rangle = \sum_i \int_{K_i \cap A} \omega, \ \  \forall \omega \in \mathcal{D}^m(\R^N) .\]

To define integer rectifiable currents with multiplicity, we recall a few notions. For $x_1, x_2, ..., x_m \in \mathbb{R}^N$, $x_1 \wedge x_2 \wedge \cdots \wedge x_m$ is a simple $m$-vector, formal linear combinations of simple $m$-vectors are $m$-vectors, and $\Lambda_m(\mathbb{R}^N)$ is the vector space of $m$-vectors. Then, as in \cite{Krantz} and \cite{Simon}, we have

\begin{definition}\textit{Let $1 \leq m \leq N$, $T$ be an {$m$-dimensional} current with finite mass. $T = \tau (\Gamma, \theta, \xi)$ is an \text{integer-multiplicity (or integral) rectifiable m-current}, if there exist $\Gamma, \theta$ and $\xi$, such that:}
\begin{enumerate}
\item \textit{$\Gamma \subset \Omega$ is an $\mathcal{H}^m$-measurable, countable $m$-rectifiable set,
with $\mathcal{H}^m(\Gamma \cap K) < \infty$ for each compact $K \subset \Omega$,}
\item \textit{$\theta$ is a locally $\mathcal{H}^m$-integrable, nonnegative, integer-valued function,}
\item \textit{$\xi: \Gamma \to \Lambda_m(\mathbb{R}^N)$ is $\mathcal{H}^m$-measurable 
such that  $\xi(x)$ is a simple unit $m$-vector in $T_x \Gamma$ for $\mathcal{H}^m$-a.e. $x \in \Gamma$, }
\item \textit{the action of $T$ {on $\mathcal{D}^m(\Omega)$} is given by
\[ \langle T, \omega \rangle = \int_\Gamma \langle \omega(x), \xi(x) \rangle \theta(x) d\mathcal{H}^m(x), 
\ \forall\omega \in \mathcal{D}^m(\Omega).
\]}
\end{enumerate}
\end{definition}
Recall an integral multiplicity one rectifiable current is an integer-multiplicity current with multiplicity $\theta \equiv 1$.

We also recall the following co-area formula and the slicing Lemma \cite[Theorem 7.6]{Krantz}.

\begin{theorem}[Coarea Formula] Given an $\mathcal{H}^m$-measurable, countable $m$-rectifiable set $\Gamma$ and a Lipschitz function $f:\mathbb{R}^N \to \mathbb{R}$, let 
$$\Gamma_+ =\Big\{ x \in \Gamma\ \big|\ T_x\Gamma,\  \nabla^\Gamma f(x) \text{ exist,\  and }\ \nabla^\Gamma f(x) \neq 0 \Big\}.$$
Then, for $L^1$-a.e. $t \in \mathbb{R}$,
 \[\Gamma_t = f^{-1}(t) \cap \Gamma_+  \textit{ is countably } \mathcal{H}^{m-1}\textit{-rectifiable}, \]
 \[\int_{- \infty}^\infty \big( \int_{\Gamma_t} g \,d\mathcal{H}^{m-1} \big)\, dt = \int_{\Gamma} \big|\nabla^\Gamma f\big|g \, d\mathcal{H}^m\]
for every nonnegative $\mathcal{H}^m$-measurable function $g$ on $\Gamma$.
\end{theorem}

\begin{definition} \textit{The slice of an integer-multiplicity $m$-current $T = \tau(\Gamma, \theta, \xi)$ \textit{by a Lipschitz function} $f$ \textit{is an integer-multiplicity} $(m-1)$\textit{-current} $\langle T, f ,t \rangle  = \tau(\Gamma_t, \theta_t, \xi_t)$, \textit{where} $\theta_t = \theta \vert_{\Gamma_t},$ \textit{and} $\xi_t = \xi \large\llcorner \frac{\nabla^\Gamma f}{|\nabla^\Gamma f|}$.}
\end{definition}

\begin{lemma}[Slicing] For every open set $W \subset \Omega$,
\[ \int_\infty^\infty \mathbf{M}(\langle T, f, t \rangle\lfloor{W} ) \,dt \leq (\text{ess} \sup_{\Gamma \cap W} 
\big|\nabla^\Gamma f\big|) \mathbf{M}(T\lfloor W). \]
\end{lemma}
 
 Recall that $T$ is called an integral current, if both $T$ and $\partial T$ are integer rectifiable. 
 Integral currents satisfy the following compactness result, see \cite{Federer60} and \cite[Lemma 6.3.11]{Simon}.

\begin{theorem}[Federer-Fleming Compactness] If $\{ T_k \}$ is a sequence of $m$-integral currents in $\Omega$ satisfying
\[ \sup_k\Big\{\mathbf{M}(T_k) +\mathbf{M}(\partial T_k)\Big\} < \infty, \]
then there exists an $m$-integral current $T$ such that after passing to a subsequence,
$T^k\rightharpoonup T$.
\end{theorem}
The compactness property yields the existence of area minimizing currents, see \cite[Lemma 7.2.1]{Simon}.

\begin{theorem} If $\Gamma$ is an $(m-1)$-integer rectifiable current {{with finite mass}}, 
then there exists an $m$-integer rectifiable current $T_0$ such that

\[ \mathbf{M}(T_0) = \inf \Big\{\mathbf{M}(T)\big|\  \partial T=\Gamma, \text{ T is m-integer rectifiable} \Big\} .\]
\end{theorem}

We recall the following monotonicity {inequality}
for  area minimizing rectifiable currents, see \cite{Federer} or \cite[Theorem 9.3]{Morgan}.

\begin{theorem}[Monotonicity] 
Let $T$ be an area-minimizing $m$-rectifiable current, and $x \in \text{spt} \; T$. Then, for $0 < r < \text{dist}(x, \text{spt} \; \partial T)$, 
\[ \frac{\mathbf{M}(T \cap B_r(x))}{\alpha_m r^m}\]
is monotonically increasing with respect to $r$, where $\alpha_m$ is the Lebesgue measure of the unit ball in $\mathbb{R}^m$.
\end{theorem}

One of the most significant results in Geometric Measure Theory is Almgren's  regularity 
{theorem} of area minimizing currents.

We define the set of regular points of {an $m$-integral rectifiable} current: 
For $x \in \text{spt }T$, $x \in \text{reg} \; T$ if there exists {a}
ball $B_r(x)$ and a positive integer $k$ such that
\[ \langle T, \omega \rangle = k \int_M \omega,\  {\ \forall \omega\in \mathcal{D}^m(B_r(x))},\]
for an embedded $m$-dimensional $C^1$ oriented submanifold $M$. 
Note that $\text{reg }T$ is a relatively open set of $\text{spt } T$.  We define the set of singular points of $T$ as $ \text{sing} \; T = \text{spt} \; T \setminus \text{reg} \; T$ which is relatively closed in $\text{spt } T$. 

We have the following partial regularity result, due to Almgren \cite{Almgren} and \cite{Lellis}.

\begin{theorem} [The Almgren Regularity] \textit{ Let $N \geq 2$, $0 < l \leq N-2$ be integers. Let $\Omega \subset \mathbb{R}^N $ be 
{a bounded, convex open set}, and $T$ be an $l$-integer rectifiable area minimizing current with finite mass in $\Omega$. Then if $l =1$,} $\text{sing } T \cap \Omega$\textit{ is empty. If $l \geq 2$}, $\text{sing }T \cap \Omega$ \textit{has 
Hausdorff dimension at most $l-2$}.
\end{theorem}

Finally, we introduce the concept of a cobordism. Its relevance to our problem is a consequence of Thom's Cobordism Theorem.

\begin{definition} \textit{Given two disjoint $n$-dimensional closed, oriented and smoothly immersed submanifolds, $M_1$ and $M_2$, a 
{cobordism} is an $(n+1)$-dimensional smoothly immersed compact submanifold $N$ such that the boundary of $N$ is the disjoint union of $M_1$ and $M_2$, in which case we say $M_1$ and $M_2$ are {cobordant}.}
\end{definition}

We can check that cobordism is an equivalence relation. The sets of manifolds cobordant to one another are cobordism classes. An $n$-dimensional closed manifold $M$ is \textit{null-cobordant} if it is cobordant with the empty manifold. In this case $M$ is the boundary of some 
$(n+1)$-dimensional manifold.

Thom's Cobordism Theorem \cite{Thom} gives a classification of the cobordism classes for each dimension. In particular, not all closed manifolds are null-cobordant, i.e. not all closed manifolds are boundaries.

\section{Cutting out a tubular neighborhood}

We start this section by restating the main Theorem \ref{equal}.

\begin{proposition} \textit{Let $\Omega \subset \mathbb{R}^N$, $N\geq 3$, be an open bounded set. 
Given an integer $l$ with $0 < l < N-2$ and $T\in \mathcal{F}_0^l$, let $A^l(T)$ and $A_0^l(T)$ be given by
\eqref{mass1} and \eqref{mass2} respectively. Then
\[ A_0^l(T) = A^l(T) .\]}
\end{proposition}

\begin{remark}  It is necessary to assume that $T$ is the boundary of some smooth manifold $M_0$, because Thom's Cobordism Theorem \cite{Thom} implies that even if $T$ were an $l$-dimensional smooth oriented manifold without boundary, it may not bound any $(l+1)$-dimensional smooth manifold. Assuming $T= \partial M_0$ ensures that $A_0^l(T)$ is not the infimum of an empty set. $M_0$ will also be crucial for the proof, because we will use it in our cut-and-paste argument.
\end{remark}

We start by applying Almgren's regularity theorem to an area-minimizing current. Let $\Gamma$ be an $(l+1)$-integer rectifiable area-minimizing current with $\partial \Gamma = T$ that realizes the infimum in the definition of $A^l(T)$. By Almgren \cite{Almgren}, 
the singular set $\mathcal{S}:= \text{sing }\Gamma$ has  Hausdorff dimension  at most $(l-1)$
{(shown in Figure 1 below)}.
{In the following lemma, we will show that there is a small $\varepsilon$-tubular neighborhood of $\mathcal{S}$ 
(shown in Figure 2 as $E$) that has both small volume and small surface area.}

\begin{lemma} Let $T, \; \Gamma, \; \mathcal{S}$ be as above. Given $\varepsilon>0$, set 
\[E_{\varepsilon} = \big \{x \in \Gamma\  \big| \ 0< \text{dist}(x,\mathcal{S})<\varepsilon \big\}, \]
\[ \partial E_\varepsilon = \big\{x \in \Gamma\ \big| \  \text{dist}(x,\mathcal{S})=\varepsilon \big\}.\]
Then there exists an $\varepsilon'\in (\frac{\varepsilon}2, \varepsilon))$ such that
\[ \mathcal{H}^{l+1}(E_{\varepsilon'}) \le C \varepsilon^2, \]
and
\[ \mathcal{H}^{l}  (\partial E_{\varepsilon'}) \le C \varepsilon, \]
where $C = C(\Gamma,l, N)>0$ depends on $\Gamma$, $l$, and $N$. 
\end{lemma}

From this point forward, $T, \Gamma, M_0$ and $\mathcal{S}$ will mean as above (Figure 1 - Note we can't actually illustrate 
$T, \Gamma, M_0, \mathcal{S}$ as above because of the restrictions on the dimensions: If we draw it for $N=3, l=1$ with $\Gamma$ being a 
$2$-integral rectifiable current, then $\mathcal{S}$ would be empty. Therefore, the sketch should be viewed only to give intuition for the higher dimensional objects).

\begin{figure}[h]
    \centering
    \begin{minipage}{0.45\textwidth}
        \centering
        \includegraphics[width=0.8\textwidth]{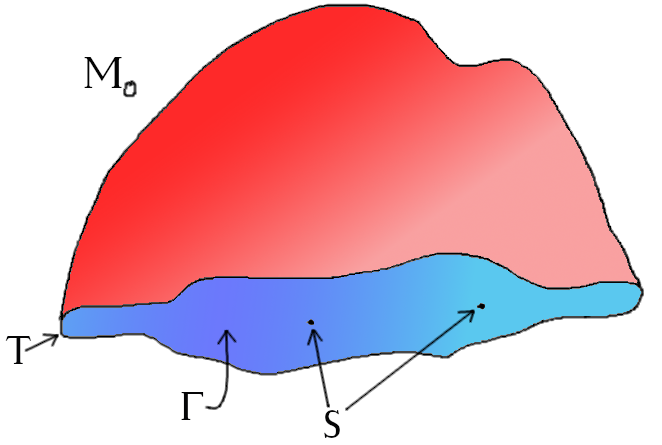} 
        \caption{Sketch for the case $N \geq 4$, $l = N-2$, $\Gamma$ is a $(l+1)$ integer rectifiable current, 
        $\mathcal{S}$ has Hausdorff dimension $(l-2)$.}
    \end{minipage}\hfill
    \begin{minipage}{0.45\textwidth}
        \centering
        \includegraphics[width=0.8\textwidth]{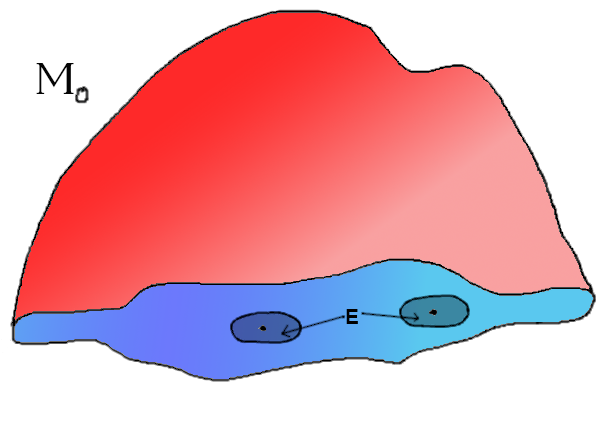} 
        \caption{An $\varepsilon$-tubular neighborhood of $\mathcal{S}$.}
    \end{minipage}
\end{figure}

 We should point out that, one may assume that $\partial E_{\varepsilon'}$ is Lipschitz hypersurface in $\Gamma$. In fact, one may replace 
 $\partial E_{\varepsilon'}$ by a finite union of smooth hypersurfaces in $\Gamma$ with transversal intersections, see 
 the remark at the end of this section.

\begin{proof} We use the slicing Lemma 2.5.
Recall $f = {\rm{dist}}(x,\mathcal{S}):\mathbb{R}^N \to \mathbb{R}$ is a Lipschitz function with $\text{Lip} \; f \leq 1$, and $|\nabla f| =1$ $ L ^N$-a.e. (see \cite[Theorem 3.14]{Evans}). The slice of $T$ by $f$ at $t$,
$\langle T, f, t\rangle$, will be the level set of the distance function $\{{\rm{dist}}(\cdot,\mathcal{S})=t\}$. Applying the slicing Lemma 2.5 
gives
\[ \int_0^\varepsilon \mathbf{M}(\langle \Gamma,f,t \rangle ) \; dt \leq  \mathbf{M}(\Gamma\lfloor E_{\varepsilon}). \]
By the mean value theorem, there exists some $\varepsilon' \in (\frac{\varepsilon}2,\varepsilon)$ such that
\[ \mathbf{M}\big(\langle \Gamma, f, \varepsilon' \rangle \big) \leq \frac{2\mathbf{M}(\Gamma\lfloor E_{\varepsilon})}{\varepsilon} .\]
We need to bound $\mathbf{M}(\Gamma\lfloor E_\varepsilon)$.
First, since $\mathcal{S}$ has  Hausdorff dimension  at most
 $(l-1)$, we know that for any small $\delta>0$, $H^{l-1+\delta}(\mathcal{S})=0$. Then by covering $\mathcal{S}$ 
 with a countable collection of balls with radius $\varepsilon$, it is implied that
\[ \sum_j \varepsilon^{l-1+\delta} \le \lambda \]
for some $\lambda>0$.

Now cover $E_\varepsilon$ with a countable collection of balls $\{B_{r_j}(x_j) \}$ with $r_j \leq \varepsilon$, so
\[ \mathbf{M}(\Gamma\lfloor E_\varepsilon) \leq \sum_j \mathbf{M}(\Gamma \cap B_{r_j}(x_j)) 
\leq \sum_j \mathbf{M}(\Gamma \cap B_\varepsilon(x_j)).\]
We will use the monotonicity Theorem 2.8 to bound $\mathbf{M}(\Gamma \cap B_\varepsilon(x_j))$. 
Let $R>1$ be {{sufficiently large}} such that $\Gamma \subset B_R(0) \subset \mathbb{R}^N$. Then
\[ \frac{\mathbf{M}(\Gamma \cap B_\varepsilon(x_j))}{\alpha_{l+1} \varepsilon^{l+1}} \le
\frac{\mathbf{M}(\Gamma \cap B_R(0))}{\alpha_{l+1}R^{l+1}} 
\le \frac{\mathbf{M} (\Gamma)}{\alpha_{l+1}}. \]
So
\begin{align*}\mathbf{M}(\Gamma\lfloor E_\varepsilon) &\leq \sum_j \mathbf{M}(\Gamma \cap B_\varepsilon(x_j))
\leq  \mathbf{M}(\Gamma) \sum_j \varepsilon^{l+1}\\
&= \mathbf{M}(\Gamma)\varepsilon^{2-\delta} \sum_j \varepsilon^{l-1+\delta} \le
\mathbf{M}(\Gamma)\varepsilon^{2-\delta}\lambda.
\end{align*}
Let $C = \mathbf{M}(\Gamma)\lambda$. Then 
\[ \mathbf{M}(\Gamma\lfloor E_\varepsilon) \le C \varepsilon^{2-\delta} .\]
Then
\[ \mathbf{M}(\langle \Gamma, f, \varepsilon' \rangle ) \leq \frac{2\mathbf{M}(\Gamma\lfloor E_\varepsilon)}{\varepsilon} 
\le C \varepsilon^{1-\delta} .\]
So taking $\delta \to 0$,
\[\textbf{M}(\langle \Gamma, f, \varepsilon' \rangle ) \leq C\varepsilon. \]
Since $E_{\varepsilon'}\setminus \mathcal{S}$ is a subset of the regular set of $\Gamma$, we have 
for $\partial E_{\varepsilon'} =\big\{x \in \Gamma\ \big|\  \text{dist}(x,\mathcal{S})=\varepsilon' \big\}$,
\[ \mathcal{H}^{l}  (\partial E_{\varepsilon'}) = \mathbf{M}(\langle \Gamma, f, \varepsilon' \rangle ) \leq C\varepsilon .\]
We also have from above
\[  \mathcal{H}^{l+1} \; (E_{\varepsilon'}) \le\mathbf{M}(\Gamma\lfloor E_{\varepsilon}) \le C \varepsilon^{2-\delta}, \]
taking $\delta \to 0$
\[ \mathcal{H}^{l+1}(E_{\varepsilon'}) \le C \varepsilon^2.\]


Note that $\partial E_{\varepsilon'}$ can be assumed to be a Lipshitz hypersurface in $\Gamma$, and hence compact ($\mathcal{S}$ is compact too) and can also be covered by finitely many small balls in the above described covering. Moreover we may assume 
$\partial E_{\varepsilon'}$ intersects with $\Gamma$ smoothly, for otherwise we replace 
$\partial E_{\varepsilon'}$ by a finite union of boundaries of these finitely small balls as said above on the out side 
so that it contains $E_{\varepsilon'}$.
\end{proof}

\section{Applying Spherical Inversion}

Given the sphere centered at the origin $0$ with radius $\varepsilon$, we define {$f: \mathbb{R}^{N+1} \to \mathbb{R}^{N+1}$ by
\[ f(x) = \frac{\varepsilon^2}{|x|^2}x, \  \  x\in \R^{N+1}.\]}
We call this function spherical inversion with respect to the sphere {$\mathbb{S}_{\varepsilon}^{N}$} centered at the origin $0$ with radius $\varepsilon$, and all spherical inversions in this section will be with respect to the same sphere.
The inversion of a sphere not intersecting the origin will be another sphere with a different center and radius 
(see \cite[Section 6.9]{Coxeter}). Therefore, the inversion of a ball not intersecting the origin will be another ball.

We need to know how diameters of spheres transform under spherical inversion.
\begin{lemma} Given a sphere of radius $r$ and center {$c \in \mathbb{R}^{N+1}$} (denoted $S_r(c)$) not intersecting the origin, the image of $S_r(c)$ under spherical inversion, $S'_{r'}(c')$, will have center $c' = \frac{\varepsilon^2}{|c|^2}c$ and radius $r' = \frac{\varepsilon^2 r}{|c|^2 - r^2}$.
\end{lemma}

\begin{proof} Clearly, $c' = f(c) = \frac{\varepsilon^2}{|c|^2}c$.
To find the new radius, consider the closest and farthest points on the sphere from the origin, i.e. the points where the ray $\overrightarrow{0c}$ intersect the sphere $S_r(c)$, $x_1$ and  $x_2$.\\
\[ r' = \frac{1}{2}|f(x_1) - f(x_2)| = \big| \frac{\varepsilon^2}{|x_1|^2}x_1 - \frac{\varepsilon^2}{|x_2|^2}x_2  \big| \]
Using $x_2 = x_1 + 2rx_1/|x_1|$ {and $|x_1|=|c|-r$}, and further simplifications we get
\[ r' = \frac{\varepsilon^2 r}{|x_1|(2r+|x_1|)} = \frac{\varepsilon^2 r}{|c|^2 - r^2} \;  \]
{This completes the proof}. 
\end{proof}

We also  recall the following  theorem, due to Jung \cite{Jung}.

\begin{theorem}[Jung's Theorem] \textit{Given a compact {$K \subset \mathbb{R}^{N+1}$, with $d = \text{diam}(K)$}, there exists a closed ball with radius 
{$r \leq d \sqrt{\frac{N+1}{2(N+2)}}$} that contains $K$}.
\end{theorem}

We would like to apply spherical inversion to the set $A = M_0 \cup T \cup \Gamma \setminus (E \cup \mathcal{S})$ (see Figure 3 below),
where $E=E_{\varepsilon'}$ is the $\varepsilon'$-tubular neighborhood of 
the singular set $\mathcal{S}$, and show that the resulting image (see Figure 4 below) has 
small $(l+1)$-dimensional Hausdorff measure.

\begin{figure}[b]
    \centering
    \begin{minipage}{0.45\textwidth}
        \centering
        \includegraphics[width=0.9\textwidth]{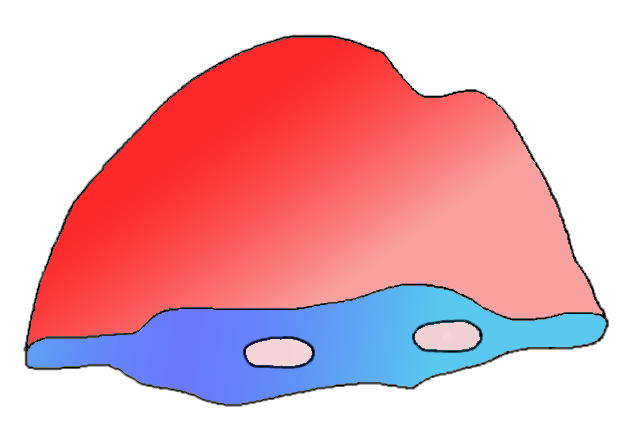} 
        \caption{Removing the tubular neighborhood of the singular set results in $A = M_0 \cup T \cup \Gamma \setminus (E \cup S)$}
    \end{minipage}\hfill
    \begin{minipage}{0.45\textwidth}
        \centering
        \includegraphics[width=0.9\textwidth]{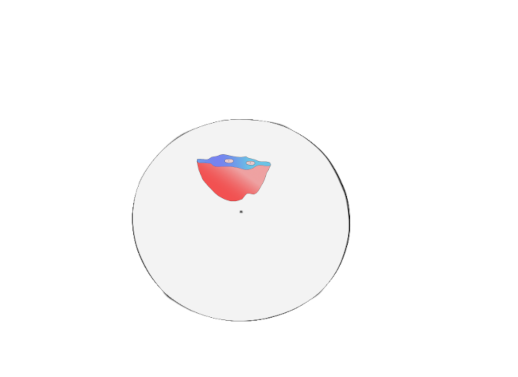} 
        \caption{A sketch of the image of $A$ after spherical inversion - not to scale}
    \end{minipage}
\end{figure}

\begin{lemma} Given $\varepsilon>0$, $A = M_0 \cup T \cup \Gamma \setminus (E' \cup \mathcal{S})$, where $E'=E_{\varepsilon'}$ is given 
by Lemma 3.3, and {let $f:\mathbb{R}^{N+1} \to \mathbb{R}^{N+1}$ }{be the spherical inversion} defined by $f(x) = \frac{\varepsilon^2}{|x|^2}x$, then
\[ \mathcal{H}^{l+1}(f(A)) \leq C \varepsilon^{2l+2} \mathcal{H}^{l+1}(A),\]
where $C = C(l,N)>0$.
\end{lemma}

\begin{proof} First, we translate the entire set $A$ so that $\text{dist} (A, O)>2$.
This will ensure that none of the balls we deal with intersect the origin.

Fix small $\delta>0$, and choose a countable collection of {sets $\{C_i \} \subset \mathbb{R}^{N+1}$} such that $\text{diam}(C_i) \leq \delta$ and  $A \subset \cup_i C_i$. By Jung's Theorem, each $C_i$ is contained in some closed ball $\overline{B}_{r_i}$ {with $r_i \leq \text{diam}(C_i)\sqrt{\frac{N+1}{2(N+2)}}$. Let $\gamma = \sqrt{\frac{N+1}{2(N+2)}}$.} So using Lemma 4.1 on each such ball,

\[ \text{diam}(f(C_i)) \leq \text{diam}(f(\overline{B}_{r_i})) = \frac{\varepsilon^2 r_i}{|c|^2 - r_i^2} \leq  \varepsilon^2 r_i \leq \text{diam}(C_i)\gamma \varepsilon^2 \leq \delta \varepsilon^2 \gamma . \]
Then, $A \subset \cup_i C_i$ implies
\[ \mathcal{H}^{l+1}_{\delta \varepsilon^2 \gamma} (f(A)) \leq \sum_i \alpha_{l+1} \big(\frac{\text{diam}(f(C_i))}{2} \big)^{l+1} 
\leq (\varepsilon^2 \gamma)^{l+1} \sum_i \alpha_{l+1}\big( \frac{\text{diam}(C_i) }{2} \big)^{l+1}. \]
Taking the infimum over all coverings $\{ C_i \}$ of $A$ with $\text{diam}(C_i) \leq \delta$ we get
\[\mathcal{H}^{l+1}_{\delta \varepsilon^2 \gamma} (f(A)) \leq (\varepsilon^2 \gamma)^{l+1} \mathcal{H}^{l+1}_\delta(A), \]
and taking $\delta \to 0$ we get
\[ \mathcal{H}^{l+1}(f(A)) \leq (\varepsilon^2 \gamma)^{l+1} \mathcal{H}^{l+1}(A). \]
Thus, we can make the measure of the image $f(A)$ as small as we like by choosing our inversion sphere to have very small radius $\varepsilon$.
\end{proof}

\section{Attaching with cones}
Given $\varepsilon,r >0$, choose $\varepsilon'\in (\frac{\varepsilon}2, \varepsilon)$, $E'=E_{\varepsilon'}$ and $\partial E'=\partial E_{\varepsilon'}$
as in the proof of Lemma 3.3, so $ \partial E' = \big\{x \in \Gamma \  \big|\    \text{dist}(x,\mathcal{S})=\varepsilon' \big\}.$ We consider the holes of 
$\Gamma \setminus \big(E' \cup \mathcal{S}\big)$, and the holes of $f\big(\Gamma \setminus (E' \cup \mathcal{S})\big)$, where $f$ represents spherical inversion through the sphere 
{$\mathbb{S}^{N}_r(0)\subset\R^{N+1}$}. We would like to connect them with conical sets (see Figure 5 below).

Define for each $t \in [0,1]$, the sets $ V_t = \big\{ tx: x \in \partial E' \big\},$ and define the cone

\[V = \Big\{ tx: x \in \partial E', t \in [0,1] \Big\} = \bigcup_{t \in [0,1]} V_t. \]
Also, define for each $x \in \partial E'$ (noting that $f(x)=\frac{r^2}{|x|^2}x$), 
\[ C_x = \Big\{ tx: t \in \big[ \frac{r^2}{|x|^2}, 1 \big] \Big\}, \]
and the truncated cone
\[ C = \bigcup_{x \in \partial E'} C_x .\]
Clearly, the geometric boundary of $C$ is $\partial E' \cup f(\partial E')$. 
The following  Lemma asserts  that the $(l+1)$-dimensional 
area of the truncated cone $C$ is small.

\begin{lemma} Given $\varepsilon, r>0$, choose $E', \partial E'$,
 $f, V,$ and $C$  as above, then
\[ \mathcal{H}^{l+1}(C) < C\varepsilon,\]
where $C = C(\Gamma, l, N)>0$ {depends on $\Gamma, l, N$}.
\end{lemma}

\begin{proof} Since $\Gamma$ is bounded, we can take $h= \sup_{x \in \partial E'} |x|<\infty$.
Then, by comparing the area of the {truncated cone 
\begin{equation}\label{cone}
C=\bigcup_{x\in \partial E'}C_x=\bigcup_{x\in \partial E'}\big\{tx: \ \frac{r^2}{|x|^2}|\le t\le 1\big\}
\end{equation}}
to the ``cylinder" of height $h$ and base $\partial E'$, we get
\[ \mathcal{H}^{l+1}(C) \leq \mathcal{H}^{l+1}(V)  \leq Ch \mathcal{H}^l(\partial E') \leq C\varepsilon. \]
This completes the proof.
\begin{figure}[t]
\centering
\includegraphics[width=0.3\textwidth]{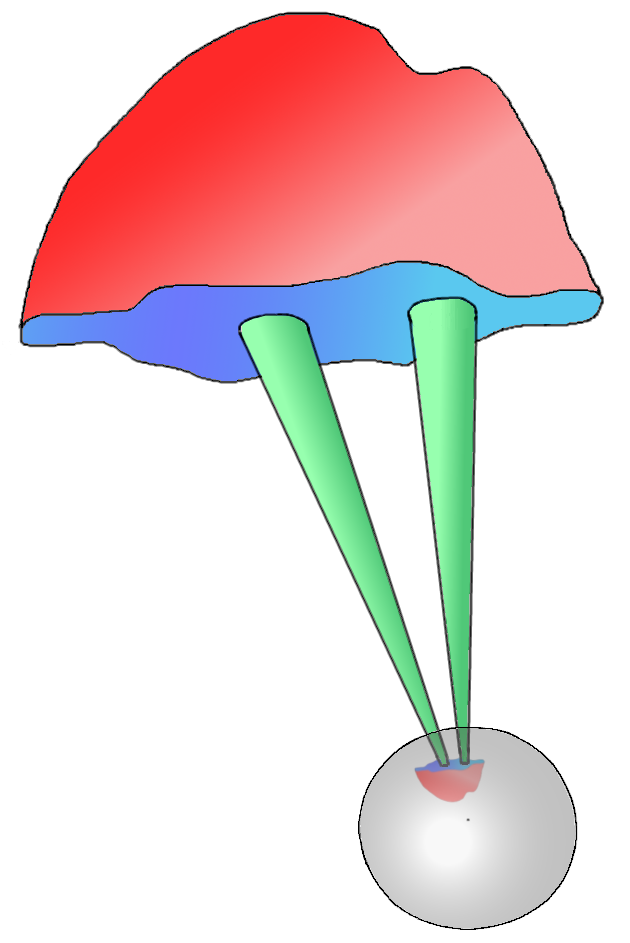} 
\caption{Attaching with cones}
\end{figure}
\end{proof}

\section{Proof of Theorem 1.1}

We start this section by recalling the statement of Theorem 1.1 again.

\begin{proposition} {\it Let $\Omega \subset \mathbb{R}^N$, $N\geq 3$, be a bounded convex open set. Given an integer $l$ with $0 < l < N-2$, suppose $T$ is an $l$-rectifiable current such that $T = \partial M_0$ for some smooth $(l+1)$-dimensional manifold $M_0 \subset \Omega\times\R$. Set 

\[ A^l(T) = \inf \Big\{ \mathbf{M}(\Gamma)\ | \Gamma \subset \Omega \text{ is an } (l+1)\text{-rectifiable current such that } \partial \Gamma = T \Big\},\]

\begin{align*} A_0^l(T) &= \inf \Big\{ |M|\ |\ M\subset\Omega\times\R \text{ is an } (l+1)\text{-smoothly immersed and oriented}\\
&\qquad\qquad\qquad\qquad\qquad\qquad\qquad\qquad\qquad\quad \text{submanifold with } \partial M = T \Big\}.
\end{align*}
Then $A_0^l(T) = A^l(T).$}
\end{proposition}
\begin{proof} Let $\Gamma$ be the minimal current that achieves the infimum $A^l(T)$. Let $\mathcal{S}$ be the set of singular points of $\Gamma$.
Fix $\varepsilon'>0$. Choose an $\varepsilon'$ tubular neighborhood $E_{\varepsilon'}$ of $\mathcal{S}$ given by Lemma 3.3 so that

\[\mathcal{H}^{l+1}(E_{\varepsilon'})<C \varepsilon^2,\]  
\[\mathcal{H}^{l}(\partial E')<C \varepsilon,\]
and the corresponding cone $D_\varepsilon$ given by Lemma 5.1 such that \[\mathcal{H}^{l+1}(D_\varepsilon)<C\varepsilon.\]

{To ensure that $M_0\cup \Gamma$ is immersed at the common boundary $T$ between $M_0$ and $\Gamma$, we slightly deform $M_0$ in the positive $x_{N+1}$-direction.
More precisely, let ${\rm{dist}}(\cdot, T):\R^{N+1}\to\R_+$ denotes the distance function to $T=\partial M_0$. Then there exists a $\delta_0>0$ such that 
${\rm{dist}}(\cdot, T)$ is smooth in $2\delta_0$-neighborhood
of $T$. Let $\eta\in C^\infty([0,\infty))$ satisfy $0\le\eta\le 1$, $\eta(t)=t$ for $0\le t\le \delta_0$, and $\eta(t)=2\delta_0$ for $t\ge 2\delta_0$.  Now we define
$\overline{M}_0=\big\{(x, x_{N+1}+\eta({\rm{dist}}((x, x_{N+1}),T)): (x, x_{N+1})\in M_0\big\}$. It is not hard to see that 
$$\mathcal{H}^{l+1}(\overline{M}_0)\le C\mathcal{H}^{l+1}(M_0).$$}

Furthermore, using Lemma 4.3, perform a spherical inversion of $A = \overline{M}_0 \cup T\cup \Gamma \setminus (E'_\varepsilon \cup \mathcal{S})$ through 
{the $N$-dimensional sphere $\mathbb{S}^N_r(0)$ with center $0$ and radius $r$}, choosing {$0<r<< \varepsilon^{\frac{1}{l+1}}$ 
sufficiently small  so that }the spherical inversion image $\widetilde{A}$ of $A$
has 
\[\mathcal{H}^{l+1}(\widetilde{A})\le C r^{2l+1} \mathcal{H}^{l+1}(A)\le C \varepsilon.\]

\smallskip
{To ensure the truncated cone $\mathcal{D}_\varepsilon \cup (\Gamma\setminus(E'\cup\mathcal{S}))$ is immersed at $\partial(\Gamma\setminus(E'\cup\mathcal{S}))$,
we also deform the truncated cone $\mathcal{D}_\varepsilon$ into $\widehat{\mathcal{D}_\varepsilon}$ 
by moving the vertex from $0\in\R^{N+1}$ to $(0',\varepsilon)$. More precisely, define 
\begin{align*}
\widehat{\mathcal{D}_\varepsilon}=\Big\{(tx, (1-t)\varepsilon)\ \big|\ x\in \partial E', \ \frac{r^2}{|x|^2} \le t\le 1\Big\}.
\end{align*}
It is clear that $\partial \widehat{\mathcal{D}_\varepsilon}=\partial E'\cup \big\{\big(\frac{r^2}{|x|^2} x, (1-\frac{r^2}{|x|^2})\varepsilon\big): x\in\partial E'\big\}$.
Also, it is not hard to see that there exists a smooth isometric map $\Phi:\mathbb R^{N+1}\to \mathbb R^{N+1}$ such that 
$$\Phi(\frac{r^2}{|x|^2} x)= \Big(\frac{r^2}{|x|^2} x, \ (1-\frac{r^2}{|x|^2})\varepsilon\Big), \ \forall x\in \partial E'.$$

Define $\widehat{A} =\Phi\big(\widetilde{A}\big)$. Then we have 
\begin{align*}\mathcal{H}^{l+1}\big(\Gamma \setminus (E_{\varepsilon'} \cup \mathcal{S}) \cup \widehat{D_\varepsilon} \cup \widehat{A}\big) 
&= \mathcal{H}^{l+1}(\Gamma) - \mathcal{H}^{l+1}(E_{\varepsilon'}) + \mathcal{H}^{l+1}(\widehat{D_\varepsilon}) 
+ \mathcal{H}^{l+1}(\widehat{A})\\
&\le \mathcal{H}^{l+1}(\Gamma) + C\mathcal{H}^{l+1}({D_\varepsilon}) 
+ C\mathcal{H}^{l+1}(\widetilde{A})\\
& \leq \mathcal{H}^{l+1}(\Gamma) + C \varepsilon
\end{align*}}
for some $C>0$ independent of $\varepsilon$.

Note that $\Gamma \setminus (E' \cup \mathcal{S}) \cup {\widehat{D_{\varepsilon}}
\cup \widehat{A}}$ is a finite union of smoothly immersed and oriented submanifolds.
In fact, $\Gamma \setminus (E'\cup \mathcal{S})$ is a smoothly immersed submanifold with boundary $\partial E'$, and
similarly it is true for $ \widetilde{A}$ (the image under the inversion with respect to the sphere described above) and the smoothly immersed cylinderical
surface {$\widehat{D_\varepsilon}$}. Let us denote $X$ as  an abstract manifold {of dimension $(l+1)$} and let $F : X\to \R^{N+1}$ be a smooth immersion
{whose image is} $\Gamma \setminus (E'\cup \mathcal{S})$,
namely, $dF$ has full rank {$(l+1)$} at every point in $X$. The immersion for $ \widetilde{A}$ would simply be the composition of $F$ 
{with the spherical} inversion map {described above}. With this,
the cylindrical surface {$\widehat{D_\varepsilon}$} can be viewed as a smooth immersion $G(x,t): \partial X \times [0,1]\to {\R^{N+1}}$.
Now  the rank of the differential of the map $G$  equals to  the rank of the differential of  the restriction $F\big|_{\partial X}$ plus one. Hence  $dG$ is  of full rank $(l+1)$.
Since the immersion $F$ {for $\Gamma\setminus (E'\cup\mathcal{S})$ and the immersion for $\widehat{D_\varepsilon}$} has the same restriction on $\partial X$, it is not hard to see one can smooth locally $F$ and that immersion for {$D_\varepsilon$ to get} a global smooth immersion of $X\cup(\partial X \times [0,1])$ into
 {$\R^{N+1}$}.  To see the latter fact, one considers $X_{\delta}$ which is the $\delta$ neighborhood of $\partial X$ in $X$. For $\delta>0$ sufficiently small,  one may view $X_{\delta}$ 
as $\partial X \times [-\delta, 0]$, and  $F$ is the smooth immersion from $X_{\delta}$  into a small open neighborhood of
$\partial E'$ into $\Gamma \setminus (E'\cup \mathcal{S}).$  For $t \in [-\delta, 0]$, 
$F_{x}$  and $F_t$ together forms a $(l+1)$ tangent frame of the image $F(X_{\delta})$, where $x \in  \partial X$
and  $l = {\rm{dim}} \ \partial E{_\varepsilon'}$.  Now what one needs to do is to extend the immersions $F$ defined on  $\partial X \times [-\delta, 0]$ and $G$ defined on $\partial X \times [\delta, 1]$ smoothly into  $\partial X \times [0, \delta]$ to become an immersion on   $\partial X \times [-\delta, 1]$ .  It is obvious what to do when $\partial X $ is a point, one simply replaces the image as a union of two smooth short arcs which intersects at a point (it could be either transversal or tangential) by a smoothly immersed slightly longer arc (in the neighborhood of the (outside) boundary points of the whole union of two arcs, one may even assume this
immersed arc are the same as the original two arcs).  For every point $x^* \in  \partial X$, one can 
{repeat what we did above}  in small neighborhood of $x^*$, and locally the resulting smooth arc changing smoothly  
when $x$ varies near $x^*$. In other words, one simply constructs a smoothly varying family (indexed by $x$) of smooth immersed arcs described as above with given given locations and its tangents at end points.  It is then by a partition of unity, one can construct such an extended  immersion on  
$\partial X \times [0, \delta]$. Similarly, one can do  smoothing near the common boundaries of {$\widehat{D_\varepsilon}$ and $\widehat{A}$}. If one does above construction with small $\delta$ so that the $(l+1)$-dimension Hausdorff measure of the image of $F (\partial X \times [0, \delta])$ is small 
(and goes to zero as $\delta$ goes to zero, and in particular much less than $\varepsilon$), then
we call the resulting approximation by $M_\varepsilon$ so that
\[ A^l_0(T) \leq \mathcal{H}^{l+1}(M_\varepsilon) \leq \mathbf{M}(\Gamma) + C\varepsilon,\]
noting that $M_\varepsilon$ satisfies $\partial M_\varepsilon= T$.
Then we can not have $A^l(T) < A_0^l(T)$, but we {have already shown $A^l(T) \leq A_0^l(T)$. Therefore,} $A^l_0(T)=A^l(T).$
\end{proof}

\section{An example}

We have shown that  when $T = \partial M_0$  is the boundary of some smooth manifold, we can approximate the area minimizing current
{$\Gamma$, with $\partial\Gamma=T$,} by smooth manifolds. We can pose the question: perhaps for such a $T$, we can always find a smooth manifold $M$ bounded by $T$ such that {the area of $M$ equals to} the area of the minimal current - in which case 
Theorem 1.1 can be stated in terms of a minimum instead of an infimum. We now show this is not the case.

Suppose we have a smooth $l$-dimensional submanifold $T_1\subset\R^N$ that has no boundary that is not null-cobordant, so it does not bound any $(l+1)$-dimensional smooth manifold. However, according to Almgren \cite{Almgren}, we know it 
{is} the boundary of a minimal current $\Gamma_1$ which has some singular set.

If we now make a copy of $T_1$ by taking a fixed $v \in \mathbb{R}^N$ with $|v| = L >> \text{diam}(T_1)$, and letting $T_2 = \{ x + v: x \in T_1 \}$, there  is a smooth orientable $(l+1)$-dimensional manifold $M_0$ such that $\partial M_0 = T_1 \cup T_2$, namely, the cylinder

\[ M_0 = \big\{ x+ tv\ \big|\  x \in T_1, t \in [0,1] \big\} .\]
Then we can again consider the quantities

\[ A^l := \inf \Big\{ \mathbf{M}(\Gamma)\big|\Gamma\subset\R^N \text{ is an } (l+1)\text{-rectifiable current such that } \partial \Gamma = T_1 \cup T_2 \Big\},\]
and
\[ A^l_0 := \inf \Big\{ |M|\ \big|\ M\subset\R^{N+1} \text{ is a } \text{smooth oriented manifold with } \partial M = T_1 \cup T_2 \Big\} .\]

Let $\Gamma_2 = \Gamma_1 + \{v\}$. Then $\partial (\Gamma_1 \cup \Gamma_2) = T_1 \cup T_2$. We note that $T_1$ and $T_2$ not being null-cobordant implies that $\Gamma_1 \cup \Gamma_2$ can't be smooth. Therefore, if there were a smooth minimal surface, it can't be the disjoint union of two disconnected components $M_1 \cup M_2$, with $T_1 = \partial M_1$ and $T_2 = \partial M_2$.

However, we also can not have a smooth minimal surface $M$, with $\partial M = T_1\cup T_2$,  that is connected. To show this, take any point $x \in T_1$, and a plane $\pi_v$ normal to $v$ that passes through the point $x +\frac{v}{2}$. If there exists some $(l+1)$-dimensional smooth minimal surface $M$ with boundary $T_1 \cup T_2$, it would intersect $\pi_v$ at some point $y $.

But, {by the area monotonicity formula in Theorem 2.8}, we have that
\[\big|M \cap B_r(y)\big| \geq \alpha_{l+1} r^{l+1}.\]
Taking $r \to \frac{L}2$, we must ahve
$$\big|M\big|\ge\big|M \cap B_{\frac{L}2}(y)\big|
\ge \alpha_{l+1} \big(\frac{L}{2}\big)^{l+1}.$$
On the other hand, it is easy to see that $\big|M\big|\le CL$. 
We get the desired contradiction, since $l\ge 1$ and $L$ can be chosen arbitrarily large.

 Although there will not be a smooth manifold that realizes the infimum $A^l$, we can still approximate the minimal current realizing $A^l$ using the method described earlier.

 \section*{Acknowledgments}

The first author is partially supported by the NSF DMS 2247773. This article is a part of the Master Thesis of the second author under the supervision of the first author. The third author is partially supported by NSF DMS 2453789 and Simons Travel Grant TSM-00007723.
The authors thank the anonymous referee for the careful reading of the manuscript and valuable comments, and for 
bringing to our attention a relevant recent preprint \cite{Almg}. 

\addcontentsline{toc}{section}{Reference}

\end{document}